\documentclass[10pt, a4paper]{amsart}
\usepackage[utf8]{inputenc}
\usepackage{amsfonts}
\usepackage{amsmath,graphics,amssymb,setspace}
\usepackage{amsthm}
\usepackage{mathrsfs}
\usepackage{graphicx}
\input{xypic}
\xyoption{all}
\usepackage{float}
\usepackage{enumerate}

\newcommand{\C}{\ensuremath{\mathbb{C}}}
\newcommand{\R}{\ensuremath{\mathbb{R}}}
\newcommand{\Z}{\ensuremath{\mathbb{Z}}}
\newcommand{\N}{\ensuremath{\mathbb{N}}}

\newcommand{\Tr}{\text{Tr}}
\newcommand{\Det}{\text{Det}}
\theoremstyle{plain}
\newtheorem{theorem}{Theorem}[section]
\newtheorem{lemma}[theorem]{\bf Lemma}
\newtheorem{prop}[theorem]{\bf Proposition}
\newtheorem*{theorem*}{Theorem}
\newtheorem{corollary}[theorem]{\bf Corollary}
\theoremstyle{definition}
\newtheorem{example}[theorem]{Example}
\newtheorem{remark}[theorem]{Remark}

\newtheorem{definition}[theorem]{Definition}
\author{Chandrasheel Bhagwat and Ayesha Fatima}

\address{Indian Institute of Science Education and Research\\ Pune\\ India.}
\email{cbhagwat@iiserpune.ac.in, ayesha.fatima@students.iiserpune.ac.in}
\date{\today}

\title[]{On The Length Spectra of Simple Regular Periodic Graphs}

\begin{document}
\begin{abstract}
	One can define the notion of primitive length spectrum for a simple regular periodic graph via counting the orbits of closed reduced primitive cycles under an action of a discrete group of automorphisms (\cite{GIL}).  We prove that this primitive length spectrum satisfies an analogue of the ‘\textit{Multiplicity one}’ property. 
	We show that if all but finitely many primitive cycles in two simple regular periodic graphs have equal lengths, then all the primitive cycles have equal lengths. 
	This is a graph-theoretic analogue of a similar theorem in the context of geodesics on hyperbolic spaces (\cite{BR}).
	We also prove, in the context of actions of finitely generated abelian groups on a graph, 
	that if the adjacency operators (\cite{CLR}) for two actions of such a group on a graph are similar, 
	then corresponding periodic graphs are length isospectral.
\end{abstract}
\maketitle

\section{Introduction}
The analogy between prime numbers and prime geodesics on a compact hyperbolic surface was considered by Selberg in his study of the counterpart of prime number theorem for compact Riemann surfaces. 
The subsequent works of Selberg\cite{SEL}, Gangolli \cite{GAN1} and Wakayama \cite{WAK} 
developed the analytic theory of zeta functions for 
hyperbolic spaces which resemble the theory L-functions for automorphic forms. 

Many mathematicians (e.g. Sunada, Chung, Ihara, Hashimoto) have studied graph theory 
with a paradigm of harmonic analysis. They have studied the zeta function associated to finite regular graphs
and developed analogies with hyperbolic spaces. 

It is well known that the Fourier coefficients of cusp forms satisfy a multiplicity one property. 
In \cite{BR}, an analogue of multiplicity one type property 
for the length spectrum of even dimensional compact hyperbolic spaces was proved and 
in \cite{KEL}, some refinements of some similar results have been proved.

One can define notions of length spectrum and primitive length spectrum for simple regular periodic graphs,
which are basically countable connected graphs with an action of a subgroup of automorphisms which satisfy some properties.  
In the first half of this paper, we establish an analogous multiplicity one property for the primitive length spectrum for these graphs.

The proof of the theorem uses the analytic properties of the Ihara zeta functions for periodic graphs studied in \cite{GIL}. 
The proof uses methods similar to the ones used in \cite{BR} and \cite{MM}.

In the later half of this paper, 
we prove a theorem about the dependence of the spectrum of adjacency operators on the length spectrum in the 
context of actions of finitely generated abelian groups on a graph. We prove that if the adjacency operators (see \cite{CLR}) 
for two actions of such a group on a graph are similar, then corresponding periodic graphs are length isospectral. 
This is a discrete analogue of the fact that the Laplace spectrum of a compact hyperbolic manifold determines 
its length spectrum (see \cite{GAN2}). 
In this context, the operators are not necessarily self-adjoint (example \ref{exmp}) and 
hence we work with a different hypothesis. 
For the proof, we again use the Ihara zeta function and its relation with length spectrum.

\section{Preliminaries}\label{sec:prelims}
 A graph $X$ consists of a non-empty set $V(X)$ of elements called the vertices of $X$ 
 together with a set $E(X)$ of unordered pairs of distinct vertices. 
 An element $\lbrace v_1, \, v_2\rbrace \, \in \, E(X)$ is called an edge connecting the vertices $v_1$ and $v_2$; 
 $v_1$ and $v_2$ are called adjacent vertices (written as $v_1 \, \sim \, v_2$). 
 A \textit{path} of \textit{length} $m$ in $X$ from $u \, \in \, V(X)$ to $v \, \in \, V(X)$ is a sequence of 
 $m+1$ vertices $(u \, = \,v_0, \, \ldots , \, v \, = \, v_m)$ such that $v_i$ is adjacent to $v_{i+1}$ for 
 $i \, = \, 0, \, \ldots , \, m$. 
 A path can also be denoted by a sequence of edges $(e_1, \, \ldots , \, e_m)$ 
 where $e_i$ is an edge connecting $v_{i \, - \, 1}$ and $v_i$. 
A path is said to be closed if $u \, = \, v$.\smallskip

\begin{definition}\cite{GIL}
A countable discrete subgroup $\Gamma$ of automorphisms of $X$ is said to act on $V(X)$
\begin{enumerate}{}{}
 \item[(1)] \textit{without inversions} if for any edge $e = \lbrace v_1, \, v_2 \rbrace$, $\nexists \, \gamma \, \in \,  \Gamma$ 
 such that $\gamma(v_1) = v_2$ and $\gamma(v_2) = v_1$ (No edge is inverted),
 \smallskip
 \item[(2)] \textit{discretely} if $\Gamma_{v} := \lbrace \gamma \in \Gamma \mid \gamma(v) = v \rbrace$ is finite,
 \smallskip
 \item[(3)] \textit{with bounded co-volume} if $vol(X/\Gamma) := \sum\limits_{v \in \mathcal{F}_0} \frac{1}{|\Gamma_v|} < \infty$, 
 where $\mathcal{F}_0$ is the complete set of representatives of the equivalence classes in $V(X)/\Gamma$,
 \smallskip
 \item[(4)] \textit{co-finitely} if $\mathcal{F}_0$ is finite.
\end{enumerate}
\end{definition}\smallskip

The finite graph $X/\Gamma$ is denoted by $B$. The vertex set of $X/\Gamma$ is $\mathcal{F}_0$ and adjacency in $X/\Gamma$ is defined by $[v] \sim [w]$ if there exist $v' \in [v] $ and $w' \in [w]$ such that $v' $ and $v' $ and $w'$ are adjacent in $X$. \smallskip

\begin{definition}[Periodic Graph] A pair $(X, \, \Gamma)$ consisting of a countable, infinite, connected
graph $X$ and a countable subgroup $\Gamma$ of the automorphism group of $X$ is called a 
\textit{periodic graph} if $\Gamma$ acts on $V(X)$ discretely and co-finitely.
\end{definition}\smallskip

\begin{definition}[Reduced Paths]
A path $(e_1, \, \ldots , \, e_m)$ is said to have \textit{backtracking} if for any $ i \, \in \lbrace 1, \, \ldots , \, m-1 \rbrace$, $e_i = e_{i+1}$
(traversed in opposite directions). A path with no backtracking is called \textit{proper}.\smallskip

A closed path is called \textit{primitive} if it is not obtained by going $n \, \geq \, 2$ times around some other closed path. \smallskip

% Here we need to add an illustrating example of tail in a path.

A proper closed path $C = (e_1, \, \ldots , \, e_m)$ is said to have a \textit{tail} if $\exists \, k \, \in \, \mathbb{N} $ such that
$e_{m-j+1} = e_{j}$ (traversed in opposite directions) for some $k$ consecutive values of $j$.  \smallskip

(For example, consider a path whose edges are $e_1, e_2, e_3 ,e_4$ where $e_1$ and $e_4$ are opposite orientations of same un-oriented edge say $e$. Here $m=4, j=1, k=1$ ) \smallskip

Proper tail-less closed paths are called \textit{reduced} closed paths. 
The set of reduced closed paths is denoted by $ \mathcal{C}$. 
\end{definition}\smallskip

A \textit{cycle} is an equivalence class of closed paths, any of which can be 
obtained from another by a cyclic permutation of vertices. 
Simply put, a cycle is a closed path with no specified starting point. 
We denote the length of the cycle $C$, which is the number of edges in any closed path representing the cycle, by $\ell(C)$.
The set of reduced cycles is denoted by $\mathcal{R}$. 
The subset consisting of primitive reduced cycles is denoted by $\mathcal{P}$. 
The primitive reduced cycles are also called prime cycles. \smallskip

\begin{definition}[$\Gamma$-Equivalence Relation]
% Two reduced closed paths \\ $C, \, D \, \in \, \mathcal{C}$ 
% are said to be \textit{$\Gamma$-equivalent}, and written as $C \sim_{\Gamma} D$, 
% if there exists an automorphism $\gamma \, \in \, \Gamma$ such that $D = \gamma (C)$. 
% The set of $\Gamma$-equivalence classes of reduced closed paths is denoted by $[\mathcal{C}]_{\Gamma}$. 
Two reduced cycles $C, \, D \, \in \, \mathcal{R}$ are said to be $\Gamma$-equivalent, and written as $C \sim_{\Gamma} D$, 
if there exists an isomorphism $\gamma \, \in \, \Gamma$ such that $D = \gamma (C)$. 
The set of $\Gamma$-equivalence classes of reduced cycles is denoted by $[\mathcal{R}]_{\Gamma}$. 
The set of $\Gamma$-equivalence classes of prime cycles is denoted by $[\mathcal{P}]_{\Gamma}$.
\end{definition}

Since $\Gamma$ is a subgroup of the automorphism group, the length of any two cycles in a 
$\Gamma$-equivalence class will be same. 
Therefore, for a $\Gamma$-equivalence class of reduced cycles $\xi \, \in \, [\mathcal{R}]_{\Gamma}$ 
we can define $\ell(\xi) \, := \, \ell(C)$ for any representative $C$ in $\xi$. \smallskip

\begin{remark}
 Another notion of length, called effective length of a cycle $C$, has been defined in \cite{GIL}.
It is defined as the length of the prime cycle underlying $C$. We do not make use of this notion in this paper.
\end{remark}\smallskip

% The \textit{effective length} of a cycle $C$, denoted by $\ell(C)$, is defined as the length of the prime cycle underlying $C$. 
% The effective length of a cycle $C$ is constant on the $\Gamma$-equivalence of $C$. 
% Therefore, for a $\Gamma$-equivalence class of reduced cycles $\xi \, \in \, [\mathcal{R}]_{\Gamma}$ 
% we can define $l(\xi) \, := \, \ell(C)$ for any representative $C$ in $\xi$. 

Consider the action of $\Gamma$ on the set of cycles.
For any cycle $C$, the \textit{stabiliser} of $C$ in $\Gamma$ is the subgroup 
$\Gamma_C \, := \, \lbrace  \gamma \, \in \, \Gamma \mid \gamma(C) = C\rbrace$. 
Also, if $C_{1}$ and $C_2$ are $\Gamma$-equivalent, say $C_{1}, \, C_2 \, \in \, \xi $, 
then $\Gamma_{C_1}$ and $\Gamma_{C_2}$ are conjugates in $\Gamma$. 
Hence $|\Gamma_{C_1}| \, = \, |\Gamma_{C_2}|$. 
This enables us to define $S(\xi)$ as the cardinality of $\Gamma_C$ for any $C$ in $\xi$. \smallskip

Henceforth, any $(X, \, \Gamma)$ is a simple, $q \, + \, 1$ regular
periodic graph such that the action of $\Gamma$ is without inversions and with bounded co-volume. 
We further assume that $\Gamma$ is such that $\Gamma_C$ and $\Gamma_v$ are trivial for any cycle $C$ and 
any vertex $v$, respectively. 
The condition of trivial stabilizer is analogous to free action of torsion-free lattices on hyperbolic spaces. \smallskip

As an example, see the the graph in Example \ref{exmp} in Section \ref{fga-groups-actions}. 

\begin{remark} It can be seen that for a cycle graph $C_m$ on $m$ vertices such that $\Z$ acts by cyclic permutations, the stabiliser of every vertex and of every cycle is $m\Z$ and hence non-trivial. Therefore, the conditions on the stabilisers that we make are not followed by these graphs. 
\end {remark} 

We define the \textit{length spectrum} of the periodic graph $(X, \, \Gamma)$ to be the function $L_{\Gamma}$ defined on 
$\mathbb{N}$ by, 
$$L_{\Gamma}(n) \, = \, \text{The number of $\xi \, \in \, [\mathcal{R}]_{\Gamma}$ such that $\ell(\xi) = n$ }.$$
%We note that  $L_{\Gamma}(1)$ and $L_{\Gamma}(2)$ are both equal to zero 
%because there can not be cycles of length $1$ and the graph is simple, respectively.

We define the \textit{primitive length spectrum} of the periodic graph $(X, \, \Gamma)$ to be the function $PL_{\Gamma}$ defined on 
$\mathbb{N}$ by, 
$$PL_{\Gamma}(n) \, = \, \text{The number of $\xi \, \in \, [\mathcal{P}]_{\Gamma}$ such that $\ell(\xi) = n$ }.$$

In this paper, we establish the following multiplicity one property for the length spectrum of periodic graphs:
\begin{theorem}\label{main}[Multiplicity One property for Primitive Length Spectrum of Regular Periodic Graphs]
Let $(X, \, \Gamma_1)$ and $(X, \, \Gamma_2)$ be two simple, $q+1$ regular periodic graphs. 
Further, assume that $\Gamma_1$ and $\Gamma_2$ act on $V(X)$ without inversions and with bounded co-volume 
and such that the stabilizer of any cycle with respect to either subgroup is trivial. 
Suppose $PL_{\Gamma_1}(n) \, = \, PL_{\Gamma_2}(n)$ for all but finitely many $n \, \in \, \mathbb{N}$. 
Then $PL_{\Gamma_1}(n) \, = \, PL_{\Gamma_2}(n)$ for all $n \, \in \, \mathbb{N}$.
Furthermore, we can conclude that $L_{\Gamma_1}(n) \, = \, L_{\Gamma_2}(n)$ for all $n \, \in \, \mathbb{N}$.
\end{theorem}\smallskip

\section{Ihara's Zeta Function for Periodic Graphs}
In this section, we consider the Ihara's zeta function for periodic graphs, as described in \cite{GIL}. 
This zeta function $Z_{X, \, \Gamma}$, an analogue of the Riemann zeta function for graphs, 
is defined as the following Euler product over the prime cycles: 

\begin{definition}[Zeta function]
$$ Z_{X, \, \Gamma}(u) \, := \, \prod\limits_{[C]_{\Gamma} \in [\mathcal{P}]_{\Gamma}} (1\, - \, u^{\ell(C)} )^{-\frac{1}{|\Gamma_{C}|}}$$ 
The above product converges for $u \in \C$ such that $|u| < \dfrac{1}{q}$. 
\end{definition}	\medskip

We now state some analytic properties of the above zeta function (See 
\cite[Theorem~2.2 (i), p.~1345, Theorem~5.1 (iii), p.~1354, Theorem~5.2 (ii), p.~1355]{GIL}).
The results in this section are valid without the assumption that $\Gamma_C$ is trivial. 	\bigskip

\begin{figure} \label{fig:Omega_q}
\begin{center}
  \includegraphics[width=0.5\textwidth]{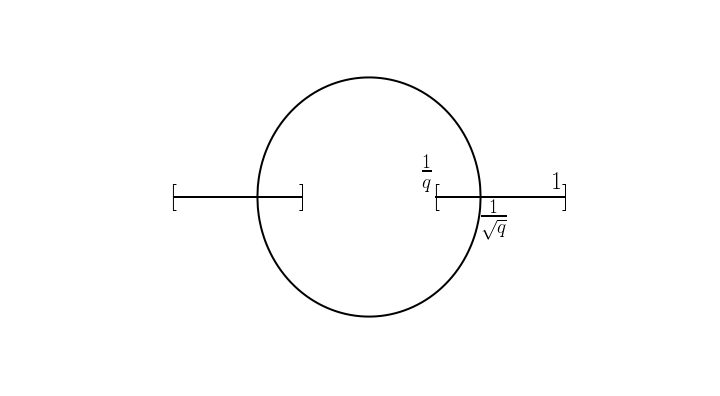}
  \caption{The set $\Omega_q$}
  \label{fig:Omega_q}
\end{center}
\end{figure}
\bigskip

\begin{theorem} Let $Z_{X, \, \Gamma}$ be the zeta function of the periodic graph $(X, \Gamma)$. Then, 
\begin{enumerate}
\item $Z_{X, \, \Gamma}(u)$ defines a holomorphic function in the open disc $\lbrace u \, \in \, \C : |u| \, < \, \frac{1}{q} \rbrace$.
\smallskip
\item We can define a completion of $Z_{X, \, \Gamma}$ which can be extended to a holomorphic function on the open set $\Omega_q$ (see figure \ref{fig:Omega_q}).
The completion $\xi_{X, \, \Gamma}$ is given by 
$$ \xi_{X, \, \Gamma}(u) \, := \, (1\, - \, u^2)^{-\chi{(B)}} (1 \, - \, u)^{|V(B)|} (1 \, - \, qu)^{|V(B)|} Z_{X, \, \Gamma}(u).$$
\smallskip
\item The function $\xi_{X, \, \Gamma}$ extends to $\Omega_q$ and satisfies the following functional equation:
$$\xi_{X, \, \Gamma}(u)  \, = \, \xi_{X, \, \Gamma}\left(\frac{1}{qu}\right).$$ 
\end{enumerate}
\end{theorem}
\medskip

Here, $\chi^2 (X)$ is the $L^2$-Euler characteristic of $(X, \, \Gamma)$, defined as 
$$\chi^2 (X) \, := 
\, \sum\limits_{v \in \mathcal{F}_0} \frac{1}{|\Gamma_v|}  \, - \,  \sum\limits_{v \in \mathcal{F}_1} \frac{1}{|\Gamma_e|}$$ 
where $\mathcal{F}_0$ is the set of representatives of equivalence classes in $V(X)/ \Gamma$ and 
$\mathcal{F}_1$ is the set of representatives of equivalence classes in $E(X)/ \Gamma$. 
It is shown in \cite[Theorem~5.1 (i), p.~1354]{GIL} that 
$$\chi^2(X)\,  = \, \chi(B)  \, = \, |V(B)| (1-q)/2,$$
where $\chi(B) := \, V(B) \, - \, E(B)$. 
It is further shown that $\chi^2(X) \, \in \, \Z$.
\medskip

The set $\Omega_q$ is an open subset of $\C$ defined as:
\begin{multline*}
\Omega_q : = \, \R^2 \, \setminus \,  \left(  \left\lbrace (x, \, y) \, \in \, \R^2 : x^2 \, + \, y^2 \, = \, \frac{1}{q} \right\rbrace \, \cup \, 
\left\lbrace (x, \, 0) \, \in \, \R^2 : \, \frac{1}{q} \, \leq \, |x| \, \leq \, 1 \right\rbrace\right).
\end{multline*}
This is the complement of the union of the circle of radius $\dfrac{1}{\sqrt {q}}$ around $0$ and the two line segments as shown in the figure \ref{fig:Omega_q}).\medskip

\section{Proof of Main Theorem}
Let $(X, \Gamma_1)$ and $(X, \Gamma_2)$ be as in section 2. 
Let $Z_{\Gamma_1}$ and $Z_{\Gamma_2}$ denote the Ihara zeta functions 
of $(X, \, \Gamma_1)$ and $(X, \, \Gamma_2)$, respectively. 
Let $\xi_{\Gamma_1}$ and $\xi_{\Gamma_2}$ be 
the extensions to $\Omega_q$ of $\Z_{\Gamma_1}$ and $Z_{\Gamma_2}$ respectively, as defined above. 
Let $B_i$ denote the finite graph $X/\Gamma_i$ for $i \, = \, 1, \, 2$.
\medskip

We use the hypothesis that $\Gamma_C$ is trivial for all cycles $C$. From the definition of the zeta function, we have in the region $|u| \, < \, \frac{1}{q}$, 
\medskip

$$\dfrac{\xi_{\Gamma_1}(u)}{\xi_{\Gamma_2}(u)} = 
\dfrac{(1 \, - \, u^2)^{-\chi(B_1)} \left[(1 \, - \, u) (1 \, - \, qu)\right]^{|V(B_1)|} \, \prod\limits_{[C]_{\Gamma_1} \in [\mathcal{P}]_{\Gamma_1}} (1\, - \, u^{\ell(C)})^{-1}}
{(1 \, - \, u^2)^{-\chi(B_2)} \left[(1 \, - \, u) (1 \, - \, qu)\right]^{|V(B_2)|} \, \prod\limits_{[C^{'}]_{\Gamma_2} \in [\mathcal{P}]_{\Gamma_2}} (1\, - \, u^{\ell(C^{'})} )^{-1}}. $$\\

Under the hypothesis of Theorem \ref{main}, all but finitely many factors in the product 
terms of the numerator and the denominator cancel out. 
Therefore, there exist finite subsets $S_1$ and $S_2$ such that in $|u| \, < \, \frac{1}{q}$, 

$$\dfrac{\xi_{\Gamma_1}\left(u\right)}{\xi_{\Gamma_2}\left(u\right)} = 
\dfrac{\left(1 \, - \, u^2\right)^{-\chi\left(B_1\right)} \left[\left(1 \, - \, u\right) \left(1 \, - \, qu\right)\right]^{|V\left(B_1\right)|} \, \prod\limits_{i \, \in S_1} \left(1\, - \, u^{\ell(C_i)}\right)^{-1}}
{\left(1 \, - \, u^2\right)^{-\chi\left(B_2\right)} \left[\left(1 \, - \, u\right) \left(1 \, - \, qu\right)\right]^{|V\left(B_2\right)|} \, \prod\limits_{j \, \in S_2} \left(1\, - \, u^{\ell(C_j)} \right)^{-1}}. $$\\

Since the product terms in the above equation are over finite indexing sets 
and the other terms are all polynomials, the above expression defines a meromorphic function on the entire complex plane.
Here we have crucially used the assumption that $\Gamma_C$ is trivial for all cycles. 
On the other hand, $\dfrac{\xi_{\Gamma_1}}{\xi_{\Gamma_2}}$ is meromorphic on $\Omega_q$ 
and hence the two expressions must agree for all $ u \, \in \, \Omega_q$. In particular, \smallskip

\begin{multline*}
\begin{split}
\dfrac{\xi_{\Gamma_1}\left(\frac{1}{qu}\right)}{\xi_{\Gamma_2}\left(\frac{1}{qu}\right)}  
& = 
\dfrac{\left(1 \, - \, \left(\dfrac{1}{qu}\right)^2\right)^{-\chi\left(B_1\right)} \left[ \left(1 \, - \, \dfrac{1}{qu}\right) \left(1 \, - \, \dfrac{1}{u}\right)\right]^{|V\left(B_1\right)|} }
{\left(1 \, - \, \left(\dfrac{1}{qu}\right)^2\right)^{-\chi\left(B_2\right)} \left[ \left(1 \, - \, \dfrac{1}{qu}\right) \left(1 \, - \, \dfrac{1}{u}\right)\right]^{|V\left(B_2\right)|}} \\
& \times  
 \dfrac{\prod\limits_{i \, \in \, S_1} \left(1\, - \, \left(\dfrac{1}{qu}\right)^{\ell(C_i)} \right)^{-1}}{\prod\limits_{j \, \in \, S_2} \left(1\, - \, \left(\dfrac{1}{qu}\right)^{\ell(C_j)}\right)^{-1}}
 \end{split}
\end{multline*}
\smallskip
\medspace

From the functional equation applied to the two zeta functions, we have for $u \, \in  \, \Omega_q$, 
$$ \dfrac{\xi_{\Gamma_1}\left(u\right)}{\xi_{\Gamma_2}\left(u\right)} = \dfrac{\xi_{\Gamma_1}\left(\frac{1}{qu}\right)}{\xi_{\Gamma_2}\left(\frac{1}{qu}\right)}$$
Since $ \chi(B_i) = |V(B_i)|(1 \, - \, q)/2$ for $i \, = \, 1, \, 2$,
we can rearrange the terms to get 
\begin{align*}
&N_1(u) \, \times \, 
\prod\limits_{j \,  \in \, S_2} \left(1\, - \, u^{\ell(C_j)} \right) \prod\limits_{i \, \in \, S_1} \left(1\, - \, \left(\frac{1}{qu}\right)^{\ell(C_i)} \right)
\\
=\,\,  &N_2(u) \, \times \, 
\prod\limits_{i \, \in \, S_1} \left(1\, - \, u^{\ell(C_i)} \right) \prod\limits_{j \, \in \, S_2} \left(1\, - \, \left(\frac{1}{qu}\right)^{\ell(C_j)} \right),
\end{align*}
where $N_1$ and $N_2$ are as follows: 

\begin{align*}
N_1(u) \, &= \,\left(1 \, - \, u^2\right)^{\left(|V\left(B_1\right)| \, -\, |V\left(B_2\right)|\right) \times \frac{(q \, - \, 1)}{2}}\left[\left(1 \, - \, u\right)\left(1 \, - \, qu\right)\right]^{|V\left(B_1\right)| \, -\, |V\left(B_2\right)|} \\[5pt]
N_2(u) \, &= \, \left(1  -  \frac{1}{q^2 u^2}\right)^{\left(|V\left(B_1\right)|  - |V\left(B_2\right)|\right) \times \frac{(q  -  1)}{2}} \left[\left(1  -  \frac{1}{u}\right) \left(1  -  \frac{1}{qu}\right)\right]^{|V\left(B_1\right)| - |V\left(B_2\right)|}
\end{align*}\\

The expressions $N_1(u)$ and $N_2(u)$ have no zeros in $\Omega_q$. 
The zero of the first product term in the expression on the left-hand side lie on a circle of radius $1$ centered at origin, 
while the zero of the second product lie on the circle of radius $\frac{1}{q}$ centered at origin. 
Similarly, the zero of the first product term in the expression on the right hand side lie on a circle of radius $1$ 
centered at origin, while the zero of the second product lie on the circle of radius $\frac{1}{q}$ centered at origin.
From the equality of the expressions, we conclude that
the zeros of the expression on the left-hand side which lie on the unit circle coincide with 
the zeros of the expression on the right-hand side which lie on the unit circle.
Hence we get equality of sets with multiplicity: 
$$ \left\{ \dfrac{2\pi k}{\ell(C_i)} \, : \, k \, \in \, \Z \, ; \, i \, \in \, S_1  \right\} \, = 
\, \left\{ \dfrac{2\pi k}{\ell(C_j)} \, : \, k \, \in \, \Z \, ; \, j \, \in \, S_2  \right\} $$
Thus we conclude that $PL_{\Gamma_1}(n) \, = PL_{\Gamma_2}(n)$ for all $n \, \in \, \mathbb{N}$. 
\medskip

Furthermore, we can check that, for $i \, = \, 1, \, 2$,
$$L_{\Gamma_i}(n) \, = \, \sum\limits_{d \vert n}PL_{\Gamma_i}\left(\frac{n}{d}\right) \, \,  \forall \, n \in \,  \mathbb{N}.$$

Every cycle of length $n$ has an underlying primitive cycle of length $\dfrac{n}{d}$ where $d | n$. Hence the above equation holds. Using this, we know that if $PL_{\Gamma_1}(n) \, = \, PL_{\Gamma_2}(n)$ for all $n \, \in \, \mathbb{N}$, then 
$L_{\Gamma_1}(n) \, = \, L_{\Gamma_2}(n)$ for all $n \, \in \, \mathbb{N}$. 
Hence $ L_{\Gamma_1}(n) \, = \, L_{\Gamma_2}(n)$ for all $n \, \in \mathbb{N}$.

\begin{flushright} $\Box$
\end{flushright}

It follows from the above discussion that $N_1(u) \, = \, N_2(u)$ for all $u \, \in \, \Omega_q$. 
Without loss of generality, we can assume that $|V(B_1)| \, \geq \, |V(B_2)|$, 
which makes $N_1$ and $N_2$ polynomial expressions.
Therefore, we can say that $N_1(u) \, = \, N_2(u)$ for all $u \, \in \, \C$. 
The zeros of $N_1(u)$ are $\{\pm 1, \, \frac{1}{q} \}$ while the zeros of $N_2(u)$ are $\{ \pm \frac{1}{q} , \, 1 \}$. 
Hence these two expressions have to be identically zero which can happen only when $|V(B_1)| \, = \, |V(B_2)|$. 
This proves the following corollary:

\begin{corollary}Suppose $X, \, \Gamma_1, \, \Gamma_2$ be as in Theorem \ref{main}. 
If $L_{\Gamma_1}(n) \, = \, L_{\Gamma_2}(n)$ for all but finitely many $n \, \in \, \mathbb{N}$ 
then $|V(X/\Gamma_1)| \, = \, |V(X/\Gamma_2)|$.
\end{corollary}
%%%%%%%%%%%%%%%%

\section{Actions of finitely generated abelian groups}\label{fga-groups-actions}
One of the important results of Ihara zeta functions of finite graphs is the rationality formula, 
which shows that the zeta function is reciprocal of a polynomial. 
This is done by expressing the zeta function as a determinant of a deformation of the graph Laplacian.
The analogous result for periodic graphs was proved in \cite{GIL} but it involves the von Neumann determinant. 

For completeness, we give a brief overview of relevant material about von Neumann algebras. 
Details and a more general treatment can be found in \cite{CLR} and \cite{GIL}.
The \textit{von Neumann algebra} of a countable discrete group $\Gamma$ 
is the algebra $\mathcal{N}(\Gamma)$ of bounded $\Gamma$-equivariant operators from $\ell^{2}(\Gamma)$ 
to $\ell^{2}(\Gamma)$. 
The \textit{von Neumann trace} of an element $f \, \in \, \mathcal{N}({\Gamma})$ is defined by 
$$\Tr_{\Gamma}(f) \, = \, \langle f(\delta_e), \, \delta_e \rangle$$ 
where $\delta_e \, \in \, \ell^2(\Gamma)$, the Kronecker delta function, is a function 
which takes the value $1$ on the identity element $e$ of $\Gamma$ and is zero elsewhere. 

For $H \, = \, \bigoplus\limits_{i \, = 1}^{n} \ell^{2}(\Gamma)$ and a bounded $\Gamma$-equivariant operator 
$T : H \longrightarrow H$, the von Neumann trace $\Tr_{\Gamma}(T)$ is defined as 
$$\Tr_{\Gamma}(T) \, = \, \sum_{i = 1}^{n} \Tr_{\Gamma}(T_{ii})$$
where $T_{ii}$ is the operator obtained by restricting the domain of $T$ to $i^{\text{th}}$ component of $H$. 
The determinant $\Det_{\Gamma}(T)$ is defined in \cite{GIL} via formal power series as 
$(\exp \, \circ \, \Tr_{\Gamma} \, \circ \, \text{Log})(T)$. 
It converges for $T$ sufficiently close to the identity operator.
A detailed analysis of this determinant can be found in \cite{GIL}.  

Let $A$ be the adjacency operator of the graph $X$, i.e., $A$ is an operator from $\ell^2(V(X))$ to itself defined by
$(Af)(v) \, = \, \sum\limits_{w \sim v}f(w)$. 
Let $Q$ be the operator from $\ell^2(V(X))$ to itself defined by $(Qf)(v) \, = \, (\text{deg}(v) \, - \, 1)f(v)$ and for $u \, \in \, \C$, let
$$\Delta(u) \, := \, I \, - \, uA \, + \, u^2Q.$$ 

We call $\Delta(u)$ the \textit{deformed Laplacian} of $(X, \Gamma)$.
Then the determinant formula (\cite[Theorem~4.1]{GIL}) shows that 
$$Z_{X, \, \Gamma}(u)^{-1} \, = \, (1 \, - u^2)^{-\chi(B)} \Det_{\Gamma}(\Delta(u)) \quad \text{for} \quad |u| \, < \, \dfrac{1}{\alpha}.$$ 

Here $\alpha \, = \, \dfrac{d\, + \, \sqrt{d^2 \, + \, 4d}}{2}$ and $d \, = \, \text{sup}_{v \, \in \, V(X)} \text{deg}(v)$. \medskip

\begin{remark} Note that in the case of $(q+1)$-regular graphs, $d = q+1$ and hence the above formula is valid for $u \in \C$ such that $|u| < \dfrac{1}{q}.$
\end{remark} \medskip

In this section, we consider the special case where $\Gamma$ is a finitely generated abelian group.
Therefore $\Gamma$  can be written as $\Z^r \times \Z_{n_1} \times \, \ldots \, \times \, \Z_{n_k}$, for some 
integers $r, n_1, \, n_2 , \, \ldots , \, n_k$ such that
$r \, \geq \, 0$, $n_i \, \geq \, 2$ for $i \,= \,  1, \ldots , \, k$, and 
$n_1 / \, n_2  / \, \ldots / \, n_k$. 
For the rest of the section, we assume as before that 
periodic graphs $(X, \, \Gamma)$ are simple and $q+1$-regular such that
 $\Gamma$ acts on $V(X)$ without inversions and with bounded co-volume.
We also assume, as in section \ref{sec:prelims}, that $\Gamma_{C}$ and $\Gamma_{v}$ are trivial for every vertex $v$ and every cycle $C$. 
The regularity assumption will give $\Delta(u) \, = \, I \, - \, uA \, + \, u^2q$. 
The zeta function $Z_{X, \, \Gamma}(u)$ is computed explicitly for the case when $\Gamma \, = \, \Z$ in \cite{CLR}. 

Let $\Z_{n_i} \, = \, \langle s_i \, \mid \, s_{i}^{n_i} \, = \, 1\rangle$, for $i \, = \, 1, \, \ldots, \, k$,  
and let $\Z^{r} \, = \, 
\langle t_1\rangle \, \times \, \langle t_2 \rangle \, \times \, \dots \, \times \, \langle t_r\rangle$. 
Suppose that the number of orbits of action of $\Gamma$ on $X$ is $n$.  
Choosing the representatives of these orbits, we can identify 
$$\ell^2(V(X)) \, = \, \bigoplus_{n} \ell^{2}(\Gamma).$$ 

Following \cite{CLR}, we describe below how the adjacency operator $A$ can be written as a $n\, \times \, n$ matrix with entries in the ring
$$R \, = \Z[t_{1}^{\pm 1}, \, t_{2}^{\pm 1}, \, \ldots , \, t_{r}^{\pm 1}, \, s_{1}, \, s_2, \, \ldots , \, s_k] / \langle s_{i}^{n_i} \, - \, 1 \, \mid \, i \, = \, 1, \, \ldots, \, k \rangle.$$  

Let $[v_1], \, [v_2], \, \ldots , \, [v_n]$ be the $\Gamma$-equivalence classes of $V(X)$. 
We look for the elements in the orbit $[v_j]$ which are adjacent to the representative vertex $v_i$ of the orbit $[v_i]$.
Any element in the orbit $[v_j]$ is of the form $\gamma \cdot v_{j}$ where 
$\gamma \, = \, \left(t_{1}^{u_1} t_{2}^{u_2} \dots t_{r}^{u_r} s_{1}^{w_1}s_{2}^{w_2} \dots s_{k}^{w_k}\right) \, \in \, G$. 
We write $A_{ij} \, = \sum \gamma$ where the sum is taken over all 
$\gamma$ such that 
$v_{i} \, \sim \, \gamma \cdot v_{j}$.  \medskip

\begin{example}\label{exmp}
Consider the $4$-regular graph $(Y, \, \Z)$ with $V(Y) \, = \, \Z \, \cup \, \Z$ as shown in figure  \ref{fig:example}.
Let $V(Y) \, = \, \lbrace \, \ldots, \, v_{-1}, \, v_0, \, v_1, \, \ldots \rbrace \, \cup \,  
\lbrace \, \ldots, \, w_{-1}, \, w_0, \, w_1, \, \ldots \rbrace$. 
The action of $\Z \, = \,  \langle t \rangle $ we consider is given by $t\cdot v_{i} \, = \, v_{i+1}$ 
and $t \cdot w_{i} \, = \, w_{i+1}$. 
This action has two orbits, namely $[v_0]$ and $[w_0]$. 
Let $A$ be the $2 \, \times \, 2$ matrix representing the adjacency operator of $(Y, \, \Z)$. 
Then

$$A \, = \, 
\begin{pmatrix}
 t\, + \, t^{-1} & 1 \, + \, t \\
 1 \, + \, t^{-1} & t \, + \, t^{-1}
\end{pmatrix}
$$
\end{example}

\begin{figure}[H]
\begin{center}
  \includegraphics[width=1\textwidth]{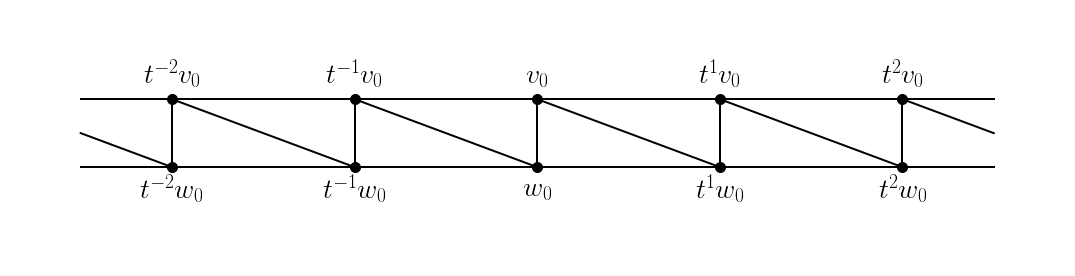}
  \caption{A $4$-regular periodic graph}
  \label{fig:example}
\end{center}
\end{figure}

Under Fourier transform, $\ell^{2}(\Z) \, = \, L^2(S^1)$ where 
$S^1 \, = \, \lbrace e^{i\theta} \, \mid \, \theta \, \in \, (-\pi, \, \pi] \rbrace$ with normalised measure. 
Therefore we have 
$$\ell^{2}(\Gamma) \, = \, \bigoplus_{r}L^2(S^1) \, \bigoplus \, \ell^{2}(\Z_{n_1}) \, \bigoplus \, \ell^{2}(\Z_{n_2}) \, 
\dots \, \bigoplus \, \ell^{2}(\Z_{n_k}).$$ 
This shows that 
$$\ell^{2}(V(X)) \, = \, \bigoplus_{nr}L^{2}(S^1) \, \bigoplus_{n} \ell^{2}(\Z_{n_1}) \, 
\bigoplus_{n} \ell^{2}(\Z_{n_2}) \, \dots \, \bigoplus_{n} \ell^{2}(\Z_{n_k}).$$
Under the Fourier transform, multiplication by $t_i$ becomes multiplication by the function 
$e^{i\theta_i}$ for $i \, = \, 1, \, 2, \, \ldots , \, r$. 
Hence $\Delta(u)$ is represented by a $n \, \times \, n$ matrix, whose entries are in terms of the variables
$\theta_1, \, \theta_2 , \, \ldots \, \theta_r, \, s_1, \,  s_2, \, \ldots , \,s_k$. 
We denote this matrix by $M_{u, \, (X, \Gamma)}$. 
(The dependence on the above variables is suppressed for ease of notation.)
To further simply the notation, we denote $M_{u, \, (X, \Gamma)}$ with $M_{u}$ when it is clear which periodic graph 
$(X, \Gamma)$ is being referred.  
Therefore we have \\

\begin{align*}
\Det_{\Gamma}(\Delta(u)) \, &= \, \text{exp} \circ \Tr_{\Gamma} \circ \text{Log}(\Delta(u)) \\
&= \, \text{exp} \sum\limits_{i = 1}^{n} \, \Tr_{\Gamma}(\text{Log}(\Delta(u)))_{ii} \\
&= \, \text{exp} \sum\limits_{i = 1}^{n} 
\left(\sum_{\Z_{n_1}} \,\sum_{\Z_{n_2}} \, \ldots\,  \sum_{\Z_{n_k}}\, 
\left( \int_{S_{1}} \int_{S_{1}} \dots \int_{S_{1}}\, (\text{Log}(M_u))_{ii} \, d\theta_1 \, \dots d\theta_r \right)\right)\\
&= \, \text{exp} 
\left(\sum_{\Z_{n_1}}\, \sum_{\Z_{n_2}} \, \ldots \, \sum_{\Z_{n_k}}\, 
\left( \int_{S_{1}} \int_{S_{1}} \dots \int_{S_{1}}\, \sum\limits_{i = 1}^{n}(\text{Log}(M_u))_{ii} \, d\theta_1 \, \dots d\theta_r \right)\right)\\
&= \, \text{exp} 
\left(\sum_{\Z_{n_1}}\, \sum_{\Z_{n_2}} \, \ldots \,  \sum_{\Z_{n_k}} \, 
\left( \int_{S_{1}} \int_{S_{1}} \dots \int_{S_{1}}\, \Tr(\text{Log}(M_u)) \, d\theta_1 \, \dots d\theta_r \right)\right)\\
&= \, \text{exp} 
\left(\sum_{\Z_{n_1}}\, \sum_{\Z_{n_2}} \, \ldots \, \sum_{\Z_{n_k}} \, 
\left( \int_{S_{1}} \int_{S_{1}} \dots \int_{S_{1}}\, \text{log}(\Det(M_u)) \, d\theta_1 \, \dots d\theta_r \right)\right)
\end{align*}\\

Note that the matrix $M_u$ in fact has entries which are polynomials in the variables 
$e^{\pm i\theta_1}, \, \ldots , \, e^{\pm i\theta_r},$ $ \,  s_1, \, \ldots , \, s_k$. 
Therefore we can say $M_u \, \in \, GL_{n}(R)$, after the change of variables 
$e^{i \theta_j} \, \longrightarrow \, t_j$, for $j \, = \, 1, \, \ldots , \, r$.

Let $(X, \, \Gamma)_1$ and $(X, \, \Gamma)_2$ be two two simple, regular periodic graphs. 
(The underlying infinite graph $X$ of the two periodic graphs is same but has different $\Gamma$-actions.)
Further, assume that both the actions of $\Gamma$ on $V(X)$ are without inversions and with bounded co-volume.
Let $M_{1, u}$ and $M_{2, u}$ denote $M_{u, \, (X, \Gamma)_1}$ and $M_{u, \,(X, \Gamma)_2}$ respectively.
Let $\Delta_{1}(u)$ and $\Delta_2(u)$ be the 
deformed Laplacians, $L_{\Gamma, \, 1}$ and $L_{\Gamma, \, 2}$ be the length spectra,  
$PL_{\Gamma, \, 1}$ and $PL_{\Gamma, \, 2}$ be the primitive length spectra
and $Z_{1}$ and $Z_{2}$ be the zeta functions 
of $(X, \Gamma)_1$ and $(X, \, \Gamma)_2$ respectively.
We can conclude the following lemma by using 
the expression of $\Det_{\Gamma_i}(\Delta_i(u))$ in terms of $M_{u, \,(X, \Gamma)_i}$. \medskip

\begin{lemma}\label{simlemma}
Suppose for a fixed $u$, the matrices $M_{1, u}$ and $M_{2, u}$ are conjugate in $GL_{n}(R)$.
Then $\Det_{\Gamma_1}(\Delta_1(u)) \, = \, \Det_{\Gamma_2}(\Delta_2(u))$.
\end{lemma}\medskip

\begin{theorem}
 Let $A_1$ and $A_2$ be the adjacency operators of the periodic graphs $(X, \Gamma)_1$ and 
 $(X, \, \Gamma)_2$ respectively. Suppose $A_1$ and $A_2$, as $n \times n$ matrices, are conjugate. 
 Then $L_{\Gamma, \, 1}(m) \, = \, L_{\Gamma, \, 2}(m)$ for all $m \, \in \, \N$. 
\end{theorem}
\begin{proof}
If $A_1$ and $A_2$ are similar, then 
$\Delta_1(u)$ is similar to $\Delta_2(u)$ for any complex number $u$ in the disc $ |u| \, < \, \frac{1}{q}$ . This follows from the fact that 
$$\Delta_{i}(u) \, := \, I \, - \, uA{i} \, + \, u^2Q.$$

In the present situation, $Q$ is just the scalar operator $qI$. Hence the operator which conjugates $A_1$ and $A_2$ also conjugates $\Delta_1(u)$ is similar to $\Delta_2(u)$.

 From the lemma \ref{simlemma}, we have $\Det_{\Gamma_1}(\Delta_1(u)) \, = \, \Det_{\Gamma_2}(\Delta_2(u))$ 
for all $u$ such that $ |u| \, < \, \frac{1}{q}$. 
From the determinant formula of the zeta function, we get $Z_1(u) \, = \, Z_2(u)$ for all 
$ |u| \, < \, \frac{1}{q}$. The proof follows from Proposition \ref{length}.
\end{proof}
\medskip

\begin{prop}\label{length}
Under the hypothesis of Theorem \ref{main} if $Z_1(u) \, = \, Z_2(u)$ then the length spectra $PL_{\Gamma, \, 1} = \, PL_{\Gamma, \, 2}$.
\end{prop}

\begin{proof}
Let 
$S_i \, = \, \lbrace  l \, \in \, \N \, \mid \,  \ell(C) \, = \, l , \, C \, \in \,  [\mathcal{P}]_{\Gamma_i} \rbrace$, 
for $i \, = \, 1,\, 2$. In other words, $S_i$ is the multiset of the lengths that occur in the primitive length spectrum of $(X, \Gamma)_i$.
Using this notation, the zeta function for the periodic graph $(X, \Gamma)_{i}$, for $i \, = \, 1, \, 2$ 
can be written as 
$$Z_{i}(u) \, = \, \prod\limits_{l \, \in \, S_i} (1\, - \, u^{l} )^{-1}.$$
Therefore, for $|u| \, < \, \frac{1}{q}$, 
\begin{align*}
\prod\limits_{l \, \in \, S_1} (1\, - \, u^{l} )^{-1} \, &= \,  \prod\limits_{m \, \in \, S_2} (1\, - \, u^{m} )^{-1}\\
\prod\limits_{l \, \in \, S_1}(1 \, + \, u^l \, + \, u^{2l} \, + \, \ldots ) \, &= \, \prod\limits_{m \, \in \, S_2}(1 \, + \, u^m \, + \, u^{2m} \, + \, \ldots )
\end{align*}
Let $l_0$ and $m_0$ be the smallest elements of $S_{1}$ and $S_2$ respectively. 
Suppose $l_0$ occurs with multiplicity $k_1$ and $m_0$ occurs with multiplicity $k_2$. 
Upon expansion, the expression of $Z_1(u)$ is of the form 
$Z_{1}(u) \, = \, (1 \, + \, k_{1}u^{l_0} \, + \; \text{higher powers of} \; u)$ 
where as the expression of $Z_2(u)$ is of the form 
$Z_{1}(u) \, = \, (1 \, + \, k_{2}u^{m_0} \, + \; \text{higher powers of} \; u)$. 
This implies that $l_0 \, = \, m_0$ and $ k_1 \, = \, k_2$. 
Using the fact that the function $(1 \, - \, u^{l_0})$ has no poles in the disc $|u| \, < \, \frac{1}{q}$, 
we can inductively conclude that $S_1 \, = \, S_2$. 
Therefore  
$PL_{\Gamma, \, 1}(m) \, = \, PL_{\Gamma, \, 2}(m)$ and hence $L_{\Gamma, \, 1}(m) \, = \, L_{\Gamma, \, 2}(m)$ for all $m \, \in \, \N$. 
\end{proof}

{\bf Acknowledgements: } The first author is partially supported by DST-SERB ECRA research award grant 30116288 for the period 2017-2020 and the SERB MATRICS grant MTR/2018/000102 for the period 2019-2022. The second author was supported by CSIR PhD fellowship for the period 2013-18.

% \begin{remark}
% In the topological case of hyperbolic surfaes, a length spectrum is defined which counts the number of 
%  geodesics, not necessarily primitive, of a given length. An analogous notion can be defined for our case. 
%  We define $ML_{\Gamma}(n)$ to be the number of $\Gamma$-equivalence classes of reduced cycles 
%  $C \, \in \, [\mathit{R}]_{\Gamma}$ such that $\ell(C) \, = \, n$. 
%  Our hypothesis for prime cycles can be shown to be equivalent to 
%  the same hypothesis for reduced cycles and it can be concluded that if 
%  $PL_{\Gamma_1}(n) \, = \, PL_{\Gamma_1}(n)$ for all $n \, \in \mathbb{N}$, then 
%  $ ML_{\Gamma_1}(n) \, = \, ML_{\Gamma_1}(n)$ for all $n \, \in \mathbb{N}$, both using the M\"{o}bius inversion formula, 
%  making our result similar to the analogous result in topological case.
% \end{remark}

\end{document}